\documentclass[11pt]{article}
\usepackage{amsmath}
\usepackage{amsthm}
\usepackage{amssymb,amsfonts}
\usepackage{hyperref}
\usepackage{latexsym}
\usepackage{todonotes}
\usepackage{nicefrac}
\usepackage{epsfig}
\usepackage{stmaryrd}
\usepackage{setspace}
\usepackage{enumerate}
\usepackage[all]{xypic}
\usepackage{bbm,ifpdf,tikz}
\ifpdf
\usepackage{pdfsync}
\fi

\newtheorem{theorem}{Theorem}[section]

\newtheorem{lemma}[theorem]{Lemma}
\newtheorem{proposition}[theorem]{Proposition}
\newtheorem{corollary}[theorem]{Corollary}

\newtheorem{conjecture}[theorem]{Conjecture}

{
\theoremstyle{definition}
\newtheorem{definition}[theorem]{Definition}
\newtheorem{example}[theorem]{Example}

\newtheorem{remark}[theorem]{Remark}
}

\newenvironment{proofclosed}{\noindent {\bf Proof of Theorem \ref{closed}:}}{\qed \par}
\newenvironment{proofz}{\noindent {\bf Proof of Theorem \ref{z-geom}:}}{\qed \par}

\newcommand{\excise}[1]{}

\renewcommand{\dim}{\operatorname{dim}}

\newcommand{\rk}{\operatorname{rk}}
\newcommand{\crk}{\operatorname{crk}}

\renewcommand{\and}{\qquad\text{and}\qquad}
\newcommand{\Ind}{\operatorname{Ind}}

\newcommand{\triv}{\operatorname{triv}}
\newcommand{\ch}{\operatorname{ch}}

\newcommand{\Z}{\mathbb{Z}}

\newcommand{\C}{\mathbb{C}}

\newcommand{\IC}{\operatorname{IC}}
\renewcommand{\H}{\operatorname{H}}
\newcommand{\IH}{\operatorname{IH}}

\newcommand{\cI}{\mathcal{I}}
\newcommand{\Ql}{\overline{\mathbb{Q}}_\ell}
\renewcommand{\a}{\alpha}
\renewcommand{\b}{\beta}

\newcommand{\pon}{(\mathbb{P}^1_k)^\cI}
\newcommand{\cO}{\mathcal{O}}

\begin{document}
\spacing{1.2}
\noindent{\LARGE\bf The \boldmath{$Z$}-polynomial of a matroid
}\\

\noindent{\bf Nicholas Proudfoot, Yuan Xu, and Benjamin Young}\\
Department of Mathematics, University of Oregon,
Eugene, OR 97403\\

{\small
\begin{quote}
\noindent {\em Abstract.} We introduce the $Z$-polynomial of a matroid, which we define in terms of the Kazhdan-Lusztig polynomial.
We then exploit a symmetry of the $Z$-polynomial to derive a new recursion for Kazhdan-Lusztig coefficients.  We solve
this recursion, obtaining a closed formula for Kazhdan-Lusztig coefficients as alternating sums of multi-indexed Whitney numbers.
For realizable matroids, we give a cohomological interpretation of the $Z$-polynomial in which the symmetry is a manifestation
of Poincar\'e duality.
\end{quote} }

\section{Introduction}
The Kazhdan-Lusztig polynomial $P_M(t)$ of a matroid $M$ was introduced by Elias, Wakefield, and the first author in \cite{EPW}.
This invariant has shown itself to be surprisingly rich, with many beautiful properties (most of them still conjectural).
For example, the coefficients of $P_M(t)$ are conjecturally non-negative; in the case where $M$ is realizable, this is proved
by interpreting the coefficients as intersection cohomology Betti numbers of the reciprocal plane of the realization \cite[Theorem 3.10]{EPW}.
The polynomial $P_M(t)$ is conjecturally log concave \cite[Conjecture 2.5]{EPW} and, even stronger, real rooted \cite[Conjecture 3.2]{kl-survey}.
Furthermore, if $M'$ is obtained from $M$ by contracting a single element, the roots of $P_{M'}(t)$ are conjectured to interlace with those of $P_M(t)$
\cite[Remark 3.5]{kl-survey}.  

If the matroid $M$ has a finite symmetry group $\Gamma$, then one can study the equivariant Kazhdan-Lusztig polynomial
$P_M^\Gamma(t)$ \cite{GPY}, whose coefficients are virtual representations of $\Gamma$ with dimension equal to the coefficients of $P_M(t)$.
In the case where $M$ is equivariantly realizable over the complex numbers, 
the same cohomological interpretation allows us to prove that the coefficients are honest representations \cite[Corollary 2.12]{GPY}.
The equivariant polynomial $P_M^\Gamma(t)$ is conjectured to be equivariantly log concave \cite[Conjecture 5.3(2)]{GPY}.

Despite all of the surprising structure that these polynomials are conjectured to have, 
very few examples are completely understood.  Kazhdan-Lusztig polynomials of 
thagomizer matroids coincide with Dyck path polynomials \cite[Theorem 1.1(1)]{thag}, and Kazhdan-Lusztig polynomials of 
fan matroids conjecturally coincide with Motzkin polynomials \cite{Gedeon-thesis}.
The equivariant Kazhdan-Lusztig coefficients of uniform matroids
have been computed explicitly \cite[Theorem 3.1]{GPY} and shown to admit the structure of finitely generated FI-modules.  
In contrast, the equivariant Kazhdan-Lusztig coefficients of braid matroids admit the structure of finitely generated $\operatorname{FS^{op}}$-modules
\cite[Theorem 6.1]{fs-braid}, and no explicit formula has appeared.  Indeed, the problem of computing Kazhdan-Lusztig coefficients of braid matroids
was the main motivation for this work.\\

In this paper we introduce the $Z$-polynomial $Z_M(t)$, which is defined as a weighted sum of the Kazhdan-Lusztig polynomials of all possible contractions of $M$.
The $Z$-polynomial is palindromic (Proposition \ref{palindrome}), reflecting the fact that, when $M$ is realizable, the coefficients of $Z_M(t)$ may be interpreted
as intersection cohomology Betti numbers of a {\em projective} variety (Theorem \ref{z-geom}), for which Poincar\'e duality holds.

Surprisingly, this symmetry of the $Z$-polynomial translates into a recursive formula for Kazhdan-Lusztig coefficients that is different from any of the recursive
formulas seen before (Corollary \ref{recurse}).  In particular, it yields a method for computing Kazhdan-Lusztig coefficients of braid matroids 
that is much faster than any previously available approach.  Furthermore, we are able to use this recursion to obtain a formula that expresses
each Kazhdan-Lusztig coefficient of $M$ as a finite alternating sum of multi-indexed Whitney numbers (Theorem \ref{closed}).
In the case of braid matroids, this becomes a finite alternating sum of products of Stirling numbers of the second kind (Corollary \ref{nice closed}).
We also obtain an equivariant version of our formula (Theorem \ref{closed-eq}), which takes a particularly nice form for uniform matroids (Proposition \ref{uni-eq}).

Our Theorem \ref{closed} bears a close resemblance to a recent result of Wakefield \cite[Theorem 5.1]{Wak}, who also obtained a formula for Kazhdan-Lusztig
coefficients as alternating sums of multi-indexed Whitney numbers.  It is likely that our formula is equivalent to Wakefield's, but the combinatorics
involved in the two formulas are very different; see Remark \ref{Wak} for further discussion of this point.\\

Our paper is structured as follows.  Section \ref{sec:def} contains the definition of the $Z$-polynomial, the proof of panlindromicity, and the recursion
for Kazhdan-Lusztig coefficients that follows from this symmetry.  Section \ref{sec:KL} uses this recursion to derive the formula for Kazhdan-Lusztig coefficients
in terms of multi-indexed Whitney numbers.  Section \ref{sec:nice} interprets these results in the case where we have a family of matroids that is closed
under contractions, such as braid matroids or uniform matroids.  
One of the results of this section is that Narayana polynomials are special cases of $Z$-polynomials (Proposition \ref{narayana}).
Section \ref{sec:roots} contains conjectures about the roots of the $Z$-polynomial,
analogous to the conjectures in \cite{kl-survey} about the roots of the Kazhdan-Lusztig polynomial.
Section \ref{sec:eq} explains how to extend our results and conjectures to the equivariant setting.

Finally, Section \ref{sec:geom} contains the cohomological interpretation of the $Z$-polynomial.  This section provides the key motivation for the definition
of the $Z$-polynomial, so in some sense it ought to appear at the very beginning of the paper.  However, the methods used Section \ref{sec:geom} are quite technical,
in contrast with the elementary and purely combinatorial methods employed in the rest of the paper, so we relegated it to the end.

\vspace{\baselineskip}
\noindent
{\em Acknowledgments:}
The authors are grateful to Sara Billey for originally suggesting the study of the $Z$-polynomial, and to Katie
Gedeon and Max Wakefield for discussions regarding the relationship between Theorem 3.3 of this paper and \cite[Theorem 5.1]{Wak}.
NP is supported by NSF grant DMS-1565036.  YX is supported by NSF grant DMS-1510296.

\section{Definition and palindromicity}\label{sec:def}
Let $M$ be a matroid on the ground set $\cI$, and let $L$ be the lattice of flats of $M$.
Given a flat $F\in L$, let $M_F$ be the {\bf localization} of $M$ at $F$; this is the matroid on the ground set $F$ whose lattice of flats is isomorphic
to $L_F := \{G\in L\mid G\leq F\}$.  Dually, let $M^F$ be the {\bf contraction} of $M$ at $F$; this is the matroid on the ground
set $\mathcal{I}\smallsetminus F$ whose lattice of flats is isomorphic to $L^F := \{G\in L\mid G\geq F\}$.
For any flat $F$, we have the {\bf rank} $\rk F := \rk M_F$ and the {\bf corank} $\crk F := \rk M^F = \rk M - \rk F$.

Let $\chi_M(t) \in\Z[t]$ be the characteristic polynomial of $M$, and let
$P_M(t) \in \Z[t]$ be the Kazhdan-Lusztig polynomial of $M$, as defined in \cite[Theorem 2.2]{EPW}.  The Kazhdan-Lusztig polynomial is characterized
by the following three properties:
\begin{itemize}
\item If $\rk M = 0$, then $P_M(t) = 1$.
\item If $\rk M > 0$, then $\deg P_M(t) < \tfrac{1}{2}\rk M$.
\item For every $M$, $\displaystyle t^{\rk M} P_M(t^{-1}) = \sum_{F}\chi_{M_F}(t) P_{M^F}(t).$
\end{itemize}


\begin{definition}\label{PtoZ}
For any matroid $M$, we define the {\bf \boldmath{$Z$}-polynomial}
$$Z_M(t) := \sum_F t^{\rk F} P_{M^F}(t).$$
\end{definition}

\begin{lemma}\label{ZtoP}
We have $$P_M(t) = \sum_F\mu(\emptyset,F)t^{\rk F} Z_{M^F}(t),$$
where $\mu:L\times L\to \Z$ is the M\"obius function.
\end{lemma}

\begin{proof}
We have
$$\sum_F\mu(\emptyset,F)t^{\rk F} Z_{M^F}(t) =
\sum_F\mu(\emptyset,F)t^{\rk F} \sum_{F\leq G} t^{\rk G - \rk F} P_{M^G}(t)
= \sum_G t^{\rk G}P_{M^G}(t) \sum_{F\leq G} \mu(\emptyset,F).$$
Since $\sum_{F\leq G} \mu(\emptyset,F) = \delta(\emptyset,G)$, this is equal to $t^{\rk\emptyset}P_{M^\emptyset}(t) = P_M(t)$.
\end{proof}

\begin{proposition}\label{palindrome}
For any matroid $M$, $Z_M(t)$ is 
palindromic of degree $\rk M$.  That is, $$t^{\rk M}Z_M(t^{-1}) = Z_M(t).$$
\end{proposition}

\begin{proof}
We have
\begin{eqnarray*}
t^{\rk M}Z_M(t^{-1}) &=& t^{\rk M} \sum_F t^{-\rk F} P_{M^F}(t^{-1})\\
&=& \sum_F t^{\rk M^F} P_{M^F}(t^{-1})\\
&=& \sum_F \sum_{F\leq G} \chi_{M^F_G}(t) P_{M^G}(t)\\
&=& \sum_G P_{M^G}(t) \sum_{F\leq G} \chi_{M^F_G}(t)\\
&=& \sum_G P_{M^G}(t)\, t^{\rk M_G}\\
&=& Z_M(t).
\end{eqnarray*}
This completes the proof.
\end{proof}

\begin{remark}
In Section \ref{sec:geom}, we will give a geometric interpretation of the $Z$-polynomial of a realizable matroid,
and in this context Proposition \ref{palindrome} can be interpreted as Poincar\'e duality (see Remark \ref{pd}).
\end{remark}

Despite the simplicity of the proof, 
Proposition \ref{palindrome} implies a previously unknown recursive formula for Kazhdan-Lusztig coefficients.
Let $c_M(i)$ and $z_M(i)$ denote the coefficients of $t^i$ in $P_M(t)$ and $Z_M(t)$, respectively.

\begin{corollary}\label{recursion}
For any matroid $M$ and natural number $i$, we have
$$c_{M}(i) = \sum_{F} c_{M^F}(\crk F - i) - \sum_{F\neq\emptyset} c_{M^F}(i-\rk F).$$
\end{corollary}

\begin{proof}
We have
$$\sum_F c_{M^F}(i-\rk F) = z_M(i) = z_M(\rk M -i) = \sum_F c_{M^F}(\crk F - i).$$
Isolating the first term in the left-hand sum, we obtain the desired equation.
\end{proof}

\begin{remark}\label{nonempty}
Suppose that $2i < \rk M$, which is a necessary condition for $c_M(i)$ to be nonzero provided that $\rk M > 0$.
Then the $F=\emptyset$
term vanishes from the first sum, and we in fact have
$$c_M(i) = \sum_{F\neq\emptyset} c_{M^F}(\crk F - i) - \sum_{F\neq\emptyset} c_{M^F}(i-\rk F).$$
Furthermore, if $i>0$, then $c_{M^F}(\crk F - i) = 0$ unless $\crk F < 2i$, which means that $\crk F - i < i$.
This tells us that our recursion expresses $c_M(i)$ in terms of other Kazhdan-Lusztig coefficeints $c_N(j)$
where $j$ is strictly smaller than $i$ {\bf and} $N$ has strictly smaller rank than $M$.
\end{remark}

\section{Kazhdan-Lusztig coefficients and Whitney numbers}\label{sec:KL}
In this section we will regard $c(i)$ as a function that takes as input a matroid and produces as output an integer.
As we observed in Remark \ref{nonempty}, the function $c(i)$ can be expressed recursively in terms of the functions
$c(0),\ldots,c(i-1)$.  If we iterate this procedure $i$ times, we obtain an expression for $c(i)$ that does not involve any 
Kazhdan-Lusztig coefficients except for $c(0)$, which is the constant function with value 1 \cite[Proposition 2.11]{EPW}.
This is exactly what we do in this section.

Given a sequence $i_r,\ldots,i_1$ of integers and a matroid $M$ with lattice of flats $L$, we define the {\bf \boldmath{$r$}-Whitney number}
$$W_M(i_r,\ldots,i_1) := \Big{|}\big\{(F_r,\ldots,F_1)\in L^r\mid F_r\leq\cdots\leq F_r\;\;\text{and $\crk F_j = i_j$ for all $j$}\big\}\Big{|}.$$
We will usually just write $W(i_r,\ldots,i_1)$, which we regard as a function that takes matroids to numbers.
For example, $W(i)$ is the function that counts the number of flats of corank $i$, while $W(i_2,i_1)$ is the function 
that counts the number of pairs of comparable flats with coranks $i_2$ and $i_1$.

\begin{remark}
Our conventions differ from the usual ones in that we index our Whitney numbers
by corank rather than rank; this will make Theorem \ref{closed} significantly simpler to state.
\end{remark}

\begin{lemma}\label{recurse}
Let $M$ be a matroid and $i_r,\ldots,i_1$ a sequence of integers.
Then $$W_M(i_r,\ldots,i_1) = \sum_{\crk F = i_r}W_{M^F}(i_{r-1},\ldots,i_{1}).$$
\end{lemma}

\begin{proof}
This is immediate from our description of the lattice of flats of $M^F$.
\end{proof}

Given positive integers $i$ and $r$ along with a subset $S\subset [r]$, let 
$$t_j(S):= \min\{\, k \mid k\geq j\;\;\text{and}\;\; k\notin S\} \in \{1,\ldots,r+1\}.$$

\begin{theorem}\label{closed}
For all $i>0$, we have
$$c(i) = \sum_{r=1}^i\;\sum_{S\subset [r]} (-1)^{|S|} \sum_{\substack{a_0<a_1<\cdots < a_r<a_{r+1}\\ a_0 = 0 \\ a_r = i\\ a_{r+1} = \rk-i}}
 W\Big(a_{t_{r}(S)} + a_{r-1},\ldots,a_{t_{1}(S)} + a_{0}\Big).$$
\end{theorem}

\begin{remark}\label{conventions}
If we try to compute $c_M(i)$ for a matroid $M$ that does not satisfy the inequality $2i<\rk M$, then the sum will
be empty, because the condition $i = a_r < a_{r+1} = \rk - i$ is not satisfied.  We will therefore obtain the number zero, which is what we expect.
Similarly, we can replace the sum over $r$ from 1 to $i$ with a sum over all $r$, because the conditions $0=a_0<\cdots < a_r = i$ can only be satisfied
if $1\leq r \leq i$.
\end{remark}

\begin{remark}
Assuming that we are evaluating this function on a matroid whose rank is greater than $2i$, the number of tuples $(a_0,\ldots,a_{r+1})$ satisfying
the given conditions is equal to the number of compositions of $i$ into $r$ parts, which is in turn equal to the binomial coefficient $\binom{i-1}{r-1}$.
Thus the total number of terms in our expression for $c(i)$ is equal to
$$\sum_{r=1}^i 2^r \binom{i-1}{r-1} = 2\sum_{s=0}^{i-1} 2^s \binom{i-1}{s} = 2(1+2)^{i-1} = 2\cdot 3^{i-1}.$$
\end{remark}

\begin{remark}\label{Wak}
Theorem \ref{closed} bears a strong similarity
to \cite[Theorem 5.1]{Wak}, where $c(i)$ is also expressed as an alternating sum of $r$-Whitney numbers.
It seems likely that there is a bijection between our index set and Wakefield's index set that makes the signed Whitney numbers in our formula match with those in his.  However, this bijection is not at all obvious; in particular, it is not even clear to us 
how to compute the size of Wakefield's index set for
general $i$.  Using a computer, Gedeon determined that the index sets do have the same size when $i\leq 4$.
\end{remark}
\vspace{\baselineskip}

\begin{proofclosed}
We induct on $i$.  When $i=1$, our formula says
$$c(1) = \sum_{S\subset [1]} (-1)^{|S|}\, W\big(a_{t_1(S)} + a_0\big).$$
We have $t_1(\emptyset) = 1$ and $t_1([1]) = 2$, so this says 
$c(1) = W(1) - W(\rk - 1)$, which was proved in \cite[Proposition 2.12]{EPW}.

Now assume that our formula holds for all $j<i$.  
Fix a matroid $M$.  By Remark \ref{nonempty}, we may assume that $2i < \rk M$, for otherwise $c_M(i) = 0$ and the sum is empty.
By Remarks \ref{nonempty} and \ref{conventions}, we have
\begin{eqnarray*}c_M(i) &=& \sum_{F\neq\emptyset} c_{M^F}(\crk F - i) - \sum_{F\neq\emptyset} c_{M^F}(i-\rk F)\\\\
&=& \sum_{F\neq\emptyset} \sum_{r}\;\sum_{S\subset [r]} (-1)^{|S|} \sum_{\substack{a_0<a_1<\cdots < a_r<a_{r+1}\\ a_0 = 0 \\ a_r = \crk F - i\\ a_{r+1} = i}}
W_{M^F}\Big(a_{t_{r}(S)} + a_{r-1},\ldots,a_{t_{1}(S)} + a_{0}\Big)\\
&-& \sum_{F\neq\emptyset} \sum_{r}\;\sum_{S\subset [r]} (-1)^{|S|} \sum_{\substack{a_0<a_1<\cdots < a_r<a_{r+1}\\ a_0 = 0 \\ a_r = i - \rk F\\ a_{r+1} = \rk M - i}}
W_{M^F}\Big(a_{t_{r}(S)} + a_{r-1},\ldots,a_{t_{1}(S)} + a_{0}\Big).
\end{eqnarray*}
We can simplify these expressions by first fixing the corank of $F$ to be some number $k$ and then applying Lemma \ref{recurse}.  This gives
us the formula
\begin{eqnarray*}c_M(i) 
&=& \sum_{r}\;\sum_{S\subset [r]} (-1)^{|S|}\sum_{k=0}^{\rk M - 1} \sum_{\substack{a_0<a_1<\cdots < a_r<a_{r+1}\\ a_0 = 0 \\ a_r = k - i\\ a_{r+1} = i}}
W_{M}\Big(k, a_{t_{r}(S)} + a_{r-1},\ldots,a_{t_{1}(S)} + a_{0}\Big)\\
&-& \sum_{r}\;\sum_{S\subset [r]} (-1)^{|S|} \sum_{k=0}^{\rk M - 1}\sum_{\substack{a_0<a_1<\cdots < a_r<a_{r+1}\\ a_0 = 0 \\ a_r = i + k - \rk M\\ a_{r+1} = \rk M - i}}
W_{M}\Big(k, a_{t_{r}(S)} + a_{r-1},\ldots,a_{t_{1}(S)} + a_{0}\Big).
\end{eqnarray*}
Next, we eliminate $k$ from both sums by observing that $k = a_{r+1}+a_r = a_{t_{r+1}(S)} + a_r$, and the inequality $k < \rk M$ turns into an inequality involving $a_r$.  
In the first sum, we get the inequality $a_r < \rk M - i$, but this is implied by the fact that $a_r < a_{r+1} = i < \rk M - i$.
In the second sum, we get the inequality $a_r < i$, which is not implied by the other conditions.  Thus we have
\begin{eqnarray*}c_M(i) 
&=& \sum_{r}\;\sum_{S\subset [r]} (-1)^{|S|} \sum_{\substack{a_0<a_1<\cdots < a_r<a_{r+1}\\ a_0 = 0 \\ a_{r+1} = i}}
W_{M}\Big(a_{t_{r+1}(S)} + a_r, a_{t_{r}(S)} + a_{r-1},\ldots,a_{t_{1}(S)} + a_{0}\Big)\\
&-& \sum_{r}\;\sum_{S\subset [r]} (-1)^{|S|} \sum_{\substack{a_0<a_1<\cdots < a_r<a_{r+1}\\ a_0 = 0 \\ a_r < i\\ a_{r+1} = \rk M - i}}
W_{M}\Big(a_{t_{r+1}(S)} + a_r, a_{t_{r}(S)} + a_{r-1},\ldots,a_{t_{1}(S)} + a_{0}\Big).
\end{eqnarray*}

We now proceed to reindex the two sums.
Given a natural number $r$ and a subset $S\subset [r]$, let $S_0 := S$ and $S_1 := S\cup\{r+1\}$, both regarded as subsets of $[r+1]$.
Then $$t_j(S_0) = \min\{\, k \mid k\geq j\;\;\text{and}\;\; k\notin S_0\} = \min\{\, k \mid k\geq j\;\;\text{and}\;\; k\notin S\} = t_j(S)$$ for all $j$, so we can replace $S$ with $S_0$ in the first sum.
On the other hand,
$$t_j(S_1) = \begin{cases} t_j(S)\;\;\text{if $t_j(S)\leq r$}\\ r+2\;\;\text{if $t_j(S)=r+1$}.\end{cases}$$
Let $b_j = a_j$ for $j\leq r$, $b_{r+1} = i$, and $b_{r+2} = \rk M - i$.  Then $a_{t_j(S)} = b_{t_j(S_1)}$, and the second sum becomes
$$\sum_{r}\;\sum_{S\subset [r]} (-1)^{|S_1|} \sum_{\substack{b_0<b_1<\cdots <b_{r+1}<b_{r+2}\\ b_0 = 0 \\ b_{r+1} = i\\ b_{r+2} = \rk M - i}}
W_{M}\Big(b_{t_{r+1}(S_1)} + b_r, b_{t_{r}(S_1)} + b_{r-1},\ldots,b_{t_{1}(S_1)} + b_{0}\Big).$$
(Note that, by replacing $(-1)^{|S|}$ with $(-1)^{|S_1|}$, we have absorbed the external minus sign.)
All together, this gives us
\begin{eqnarray*}c_M(i) 
&=& \sum_{r}\;\sum_{S\subset [r]} (-1)^{|S_0|} \sum_{\substack{a_0<a_1<\cdots < a_{r+1}\\ a_0 = 0 \\ a_{r+1} = i}}
W_{M}\Big(a_{t_{r+1}(S_0)} + a_r, a_{t_{r}(S_0)} + a_{r-1},\ldots,a_{t_{1}(S_0)} + a_{0}\Big)\\
&+& \sum_{r}\;\sum_{S\subset [r]} (-1)^{|S_1|} \sum_{\substack{b_0<b_1<\cdots < b_{r+1}<b_{r+2}\\ b_0 = 0 \\ b_{r+1} = i\\ b_{r+2} = \rk M - i}}
W_{M}\Big(b_{t_{r+1}(S_1)} + b_r, b_{t_{r}(S_1)} + b_{r-1},\ldots,b_{t_{1}(S_1)} + b_{0}\Big).
\end{eqnarray*}
Finally, we observe that summing over all subsets $S\subset [r]$ and then separately considering $S_0$ and $S_1$ is the same
as summing over all subsets of $[r+1]$.  If we now re-index the outer sum by letting $s=r+1$, we obtain the desired formula for $c_M(i)$, and the induction is complete.
\end{proofclosed}

\section{Nice families}\label{sec:nice}
Given two matroids $M$ and $M'$, we will write $M\simeq M'$ if $M$ and $M'$ have isomorphic simplifications,
or (equivalently) if they have isomorphic lattices of flats.  Since the Kazhdan-Lusztig polynomial is defined in terms
of the lattice of flats, we have $P_{M}(t) = P_{M'}(t)$ whenever $M\simeq M'$.

We define a {\bf nice family}  
to be a sequence of matroids $\{M_d\mid d\geq 0\}$
with the property that $\rk M_d = d$ and, for any corank $k$ flat $F$ of $M_d$, we have $M_d^F \simeq M_{k}$.
Examples of nice families include the following.
\begin{enumerate}
\item $M_d$ is the braid matroid of rank $d$.  Equivalently, this is the matroid associated with the complete graph on $d+1$ vertices,
or the matroid associated with the Coxeter arrangement of type $A_d$.  
\item $M_d$ is the matroid associated with the Coxeter arrangement of type $B_d$.  
\item $M_d = U_{m,d}$ is the uniform matroid of rank $d$ on $m+d$ elements, where $m$ is fixed.
\item 
$M_d$ is the matroid represented by all vectors in $\mathbb{F}_q^d$, where $q$ is a fixed prime power.
\end{enumerate}

\begin{remark}
The matroids associated with Coxeter arrangements of type D do not form a nice family.
For such a matroid, the contraction of a flat of rank 1 is no longer a matroid associated with any Coxeter arrangement.
\end{remark}

Fix a nice family.  For ease of notation, we will write $P_d = P_{M_d}$, $Z_d = Z_{M_d}$, $W_d = W_{M_d}$, and $c_d = c_{M_d}$.
Recall that $W_{d}(k)$ is the number of flats of $M_d$ of corank $k$,
and let $w_d(k) := \sum_{\crk F = k}\mu(\emptyset,F)$ be the coefficient of $t^{k}$ in the characteristic polynomial of $M_d$.
Then Definition \ref{PtoZ} and Lemma \ref{ZtoP} tell us that
\begin{equation}\label{PZ} Z_d(t) = \sum_{k=0}^{d} W_d(k) t^{d-k} P_{k}(t)
\and
P_d(t) = \sum_{k=0}^{d} w_d(k) t^{d-k} Z_{k}(t).\end{equation}
In the four families described above, we have the following.
\begin{enumerate}
\item For the braid matroid, $W_d(k) = S(d+1,k+1)$ and $w_d(k) = s(d+1,k+1)$ are Stirling numbers of the second and first kind,
respectively.
\item For the matroid associated with the type $B_d$ Coxeter arrangement,
$$W_k(d) = \sum_{j=k}^d 2^{j-k}\binom{d}{j} S(j,k)
\and w_k(d) = (-1)^{d-k}\sum_{j=k}^d (-2)^{d-j}\binom j k s(d,j).$$
The first formula appears in \cite[Proposition 3]{Suter}.  The second appears in \cite[Sequence A028338]{oeis}, using the fact that the exponents of 
this arrangement are $1, 3, \ldots, 2d-1$.
\item For the uniform matroid $U_{m,d}$,
$$W_d(k) = \begin{cases}\binom{d+m}{k+m} \;\;\,\text{if $k>0$}\\ 
1 \;\;\;\;\;\;\text{if $k=0$}
\end{cases}
\text{and}\;\;\;
w_d(k) = \begin{cases}
(-1)^{d-k}\binom{d+m}{k+m} \qquad\;\;\; \text{if $k>0$ or $d=k=0$}\\
\sum_{j=0}^m (-1)^{d+j} \binom{d+m}{d+j} \;\; \text{if $d>k=0$}.
\end{cases}
$$
\item For the matroid represented by all vectors in $\mathbb{F}_q^d$, $W_d(k) = \displaystyle\binom{d}{k}_{\!\!q}$ and $w_d(k) = (-1)^{d-k}q^{\binom{d-k}{2}}\displaystyle\binom{d}{k}_{\!\!q}.$\\
\end{enumerate}
Corollary \ref{recursion} and Remark \ref{nonempty} translate to the following statement.

\begin{corollary}\label{nice linear}
If $2i < d$, then
$$c_{d}(i) = \sum_{k=0}^{d-1} W_d(k) c_{k}(k-i) - \sum_{k=0}^{d-1} W_d(k) c_{k}(i-d+k).$$
\end{corollary}

\begin{remark}
Corollary \ref{nice linear} has proved to be faster than any previously known formula for computing the Kazhdan-Lusztig coefficients
of the braid matroid.
\end{remark}

We may also interpret $r$-Whitney numbers in terms of the numbers $W_d(k)$.  
The following result follows from Lemma \ref{recurse}.

\begin{corollary}\label{whitney stirling}
If we set $i_{r+1} := d$, then we have
$$W_d(i_r,\ldots,i_1) = \prod_{j=1}^{r} W_{i_{j+1}}(i_{j}).$$
\end{corollary}

Combining Corollary \ref{whitney stirling} with Theorem \ref{closed}, we obtain the following result.

\begin{corollary}\label{nice closed}
We have
$$c_d(i)\;\; =\;\; \sum_{r=1}^i\;\sum_{S\subset [r]} (-1)^{|S|} \sum_{\substack{0<a_1<\cdots < a_r<a_{r+1}\\ a_0 = 0\\ a_r = i\\ a_{r+1} = \rk - i}} 
\;\prod_{j=1}^{r} W_{a_{t_{j+1}(S)} +a_{j}}\big(a_{t_{j}(S)} +a_{j-1}\big).$$
\end{corollary}

Given a nice family, it is natural to use generating functions to collect the Kazhdan-Lusztig polynomials and the $Z$-polynomials.
Let $$P(t,u) := \sum_{d=0}^\infty P_d(t) u^d
\and Z(t,u) := \sum_{d=0}^\infty Z_d(t) u^d.$$
We will also be interested in the exponential generating functions
$$\tilde P(t,u) := \sum_{d=0}^\infty P_d(t) \frac{u^d}{d!}
\and \tilde Z(t,u) := \sum_{d=0}^\infty Z_d(t) \frac{u^d}{d!}.$$
In addition, consider the generating functions
$$g_k(x) := \sum_{d=k}^\infty w_d(k) x^d\and G_k(x) := \sum_{d=k}^\infty W_d(k) x^d,$$
along with their exponential analogues 
$$\tilde g_k(x) := \sum_{d=k}^\infty w_d(k) \frac{x^d}{d!}\and \tilde G_k(x) := \sum_{d=k}^\infty W_d(k) \frac{x^d}{d!}.$$ 

\begin{proposition}\label{crazy}
We have $$P(t,u) = \sum_{k=0}^\infty t^{-k} Z_k(t) g_k(tu)\and Z(t,u) = \sum_{k=0}^\infty t^{-k} P_k(t) G_k(tu),$$
and also $$\tilde P(t,u) = \sum_{k=0}^\infty t^{-k} Z_k(t) \tilde g_k(tu)\and \tilde Z(t,u) = \sum_{k=0}^\infty t^{-k} P_k(t) \tilde G_k(tu).$$
\end{proposition}

\begin{proof}
We have
$$\sum_{k=0}^\infty t^{-k} Z_k(t) g_k(tz) = \sum_{k=0}^\infty t^{-k} Z_k(t) \sum_{d=k}^\infty w_d(k) (tu)^d
= \sum_{d=0}^\infty \sum_{k=0}^{d} w_d(k)t^{d-k} Z_{k}(t) u^d,$$
which is equal to $P(t,u)$ by Equation \eqref{PZ}.
The proofs of the other three statements are identical.
\end{proof}

\begin{example}\label{braid}
In type A (the first example), Proposition \ref{crazy} is most elegant in its exponential version.
We have $$\tilde g_k(x) = \frac{1}{1+x}\frac{\log(1+x)^k}{k!} \and \tilde G_k(x) = e^x\frac{(e^x-1)^k}{k!},$$
so Proposition \ref{crazy} says that
$$\tilde P(t,u) = \frac{1}{1+tu}\sum_{k=0}^\infty \frac{\log(1+tu)^k}{t^k} \frac{Z_k(t)}{k!}\and \tilde Z(t,u) = e^{tu}\sum_{k=0}^\infty \frac{(e^{tu}-1)^k}{t^k} \frac{P_k(t)}{k!}.$$
\end{example}

\begin{example}\label{type B}
In type B (the second example), we have
$$\tilde g_k(x) = \frac{1}{\sqrt{1+2x}}\frac{\log(1+2x)^k}{2^k\, k!} \and \tilde G_k(x) = e^x\frac{(e^{2x}-1)^k}{2^k\, k!},$$
so Proposition \ref{crazy} says that
$$\tilde P(t,u) = \frac{1}{\sqrt{1+tu}}\sum_{k=0}^\infty \frac{\log(1+2tu)^k}{(2t)^k} \frac{Z_k(t)}{k!}\and \tilde Z(t,u) = e^{tu}\sum_{k=0}^\infty \frac{(e^{2tu}-1)^k}{(2t)^k} \frac{P_k(t)}{k!}.$$
\end{example}

We next consider the third example when $m=1$, so that $M_d$ is the uniform matroid of rank $d$ on $d+1$ elements.
In this case, we can use Proposition \ref{crazy} to derive a precise formula for the $Z$-polynomial.

\begin{proposition}\label{narayana}
If $M_d$ is the uniform matroid of rank $d$ on $d+1$ elements, then the coefficient $z_d(i)$ of $t^i$ in $Z_d(t)$
is equal to the Narayana number $N(d+1,i+1) = \frac{1}{d+1}\binom{d+1}{i+1}\binom{d+1}{i}$.
\end{proposition}

\begin{proof}
We have $$G_k(x) = \sum_{d=k}^\infty \binom{d+1}{k+1}x^d = \frac{x^k}{(1-x)^{k+2}}$$ if $k>0$,
and $$G_0(x) = \sum_{d=k}^\infty x^d = \frac{1}{1-x}.$$
Proposition \ref{crazy} therefore tells us that
\begin{eqnarray*}Z(t,u) &=& \sum_{k=0}^\infty t^{-k} P_k(t) G_k(tu)\\
&=& \frac{1}{1-tu} + \frac{1}{(1-tu)^2}\sum_{k=1}^\infty P_k(t) \left(\frac{u}{1-tu}\right)^k\\
&=& \frac{1}{1-tu} + \frac{1}{(1-tu)^2}\left(P\!\left(t,\frac{u}{1-tu}\right)-1\right).
\end{eqnarray*}
In \cite[Section 2]{PWY}, we showed that $$P(t,v) - 1 = 
\frac{2}{v}\cdot
\frac{{\left(2tv + 1\right)} v - 1 + \sqrt{1-2 \, {\left(2tv + 1\right)} v + v^{2}}}{{1-\left(2tv + 1\right)}^{2}}.$$
Setting $v = \frac{u}{1-tu}$, we obtain an explicit algebraic expression for $Z(t,u)$.
On the other hand, it is shown in \cite[Equation (2.6)]{Petersen-Narayana} that
$$\sum_{d=0}^\infty \sum_{i=0}^\infty N(d+1,i+1) t^i u^d = - \frac{1}{u} + \frac{1 + u(t-1) - \sqrt{1-2u(t+1)+u^2(t-1)^2}}{2tu^2}.$$
It is an elementary exercise to check that this formula coincides with our expression for $Z(t,u)$.
\end{proof}

\begin{example}\label{favorite}
Finally, we consider the fourth example, where $M_d$ is the matroid represented by all vectors in $\mathbb{F}_q^d$.
This matroid is modular, so we have $P_d(t) = 1$ for all $d$ \cite[Proposition 2.14]{EPW}.
It follows that $$Z_d(t) = \sum_{k=0}^d W_{d}(d-k) t^k = \sum_{k=0}^d \binom{d}{k}_{\!\!q} t^k.$$
\end{example}

\section{Roots of the \boldmath{$Z$}-polynomial}\label{sec:roots}
In \cite[Conjecture 3.2]{kl-survey}, we conjectured that the polynomial $P_M(t)$ is real rooted.  Here
we make the analogous conjecture for the $Z$-polynomial.

\begin{conjecture}\label{real}
For any matroid $M$, all of the roots of $Z_M(t)$ lie on the negative real axis.
\end{conjecture}

We also gave a conjectural relationship between the roots of $P_M(t)$ and the roots of a contraction of $P_{M/e}(t)$, where $e\in\cI$
is a non-loop of $M$ \cite[Conjecture 3.3]{kl-survey}, assuming certain nondegeneracy conditions.  Here we make a similar conjecture for $Z$-polynomials,
but rather than attempting to formulate the correct nondegeneracy conditions, we focus on the case of a nice family, where the conjecture takes a particularly
clean form.  If $f(t)$ is a polynomial of degree $d$ with roots $\a_1\leq\cdots\leq\a_{d}$ and $g(t)$ is a polynomial of degree $d-1$ with roots $\b_1\leq\cdots\leq\b_{d-1}$,
we say that $f(t)$ {\bf interlaces} $g(t)$ if $\a_i \leq \b_i \leq \a_{i+1}$ for all $0<i<d$.  If the inequalities are strict, we say that $f(t)$ {\bf strictly interlaces} $g(t)$.

\begin{conjecture}\label{real-nice}
If $\{M_d\mid d\geq 0\}$ is a nice family, then for all $d$, $Z_{d}(t)$ interlaces $Z_{d-1}(t)$.
\end{conjecture}

\begin{example}\label{sturm}
Suppose that $M_d$ is the uniform matroid of rank $d$ on $d+1$ elements.  We showed in Proposition \ref{narayana}
that $Z_d(t) = \sum_{i=0}^d N(d+1,i+1)t^i$ is a Narayana polynomial, 
and these polynomials are known to have interlacing negative real roots
\cite[Problem 4.7]{Petersen-Narayana}.  Thus Conjectures \ref{real} and \ref{real-nice} hold for this nice family.
\end{example}

\begin{remark}
It is interesting to compare the state of affairs for the Kazhdan-Lusztig polynomials and the $Z$-polynomials
of the matroids in Example \ref{sturm}.  The Kazhdan-Lusztig polynomials are known to have negative real roots
\cite[Theorem 3.3]{kl-survey}, but the interlacing property for Kazhdan-Lusztig polynomials \cite[Conjecture 3.4]{kl-survey}
is still open, even in this simple example.  
\end{remark}

\begin{proposition}\label{q-real-nice}
Fix a prime power $q$.  If $M_d$ is the matroid represented by all vectors in $\mathbb{F}_q^d$, then Conjectures \ref{real} and \ref{real-nice}
hold for the nice family $\{M_d\mid d\geq 0\}$.
\end{proposition}

\begin{proof}
We will prove a slightly stronger statement by induction on $d$.  
We will prove that, for every $d$, $Z_d(t)$ has roots $\a_1,\ldots,\a_d<0$ with $\a_i < q\a_{i+1}$ for all $0<i<d$, and that $Z_{d}(t)$
strictly interlaces $Z_{d-1}(t)$. The statement is trivial when $d=1$.

Assume that $Z_{d-1}(t)$ has roots $\b_1,\ldots,\b_{d-1}<0$ with $\b_i < q\b_{i+1}$ for all $0<i<d-1$.
Since $Z_{d-1}(t)$ has $d-1$ distinct real roots, it changes sign at each root.
Since $\b_{i-1} < q\b_{i} < \b_{i}$ for all $1<i<d$ and $Z_{d-1}(0) = 1$, this implies that $(-1)^dZ_{d-1}(q\b_i)$ is positive when $i$ is even and negative when $i$ is odd.

As observed in Example \ref{favorite}, we have
$$Z_d(t) = \sum_{k=0}^d \binom{d}{k}_{\!\!q} t^k.$$
Using the identity
\begin{equation*}\label{pascal}\binom{d}{k}_{\!\!q} = \binom{d-1}{k}_{\!\!q}q^k + \binom{d-1}{k-1}_{\!\!q},\end{equation*}
this implies that $$Z_d(t) = Z_{d-1}(qt) + tZ_{d-1}(t).$$
In particular, we have $$Z_d(\b_i) = Z_{d-1}(q\b_i) + tZ_{d-1}(\b_i) = Z_{d-1}(q\b_i).$$
This tells us that the numbers $Z_d(\b_i)$ alternate in sign, and therefore that for all $1<i<d$ there exists a root $\a_i$ of $Z_d(t)$ with $\a_i\in(\b_{i-1},\b_i)$.
In addition, we know that 
$Z_d(\b_{d-1}) = Z_{d-1}(q\b_{d-1}) < 0$ but $Z_d(0) = 1$, so there must exist a root $\a_d$ of $Z_d(t)$ with $\b_{d-1}<\a_d<0$.
Similarly, we know that
$(-1)^dZ_d(\b_1) = (-1)^dZ_{-1}(q\b_1) < 0$ but $(-1)^dZ_d(t)$ is positive for $t$ sufficiently negative, so there must exist
a root $\a_1<\b_1$.  This proves that the roots of $Z_d(t)$ lie on the negative real axis and $Z_d(t)$ strictly interlaces $Z_{d-1}(t)$.

To complete the induction, we still need to prove that $\a_i < q\a_{i+1}$ for all $0<i<d$.
For all such $i$, we have
$$0 = Z_d(\a_i) = Z_{d-1}(q\a_i) + \a_iZ_{d-1}(\a_i)$$
and
$$0 = Z_d(\a_{i+1}) = Z_{d-1}(q\a_{i+1}) + \a_{i+1}Z_{d-1}(\a_{i+1}).$$
We know that $\a_iZ_{d-1}(\a_i)$ and $\a_{i+1}Z_{d-1}(\a_{i+1})$ have opposite signs, therefore so do $Z_{d-1}(q\a_i)$ and $Z_{d-1}(q\a_{i+1})$.
It follows there there is a root $\b_{j_i}$ of $Z_{d-1}(t)$ in between $q\a_i$ and $q\a_{i+1}$.  Since $\b_{j_1} < \cdots < \b_{j_{d-1}}$, we must have $j_i = i$,
and therefore $\a_i < \b_i < q\a_{i+1}$.
\end{proof}

\begin{remark}\label{modular}
We have proved Conjectures \ref{real} and \ref{real-nice} for our third family when $m=1$ (Example \ref{sturm})
and for our fourth family (Proposition \ref{q-real-nice}).  For the first two families, and for the third family when $2\leq m\leq 10$,
we have checked the conjectures on a computer for all $d\leq 30$.
\end{remark}

\section{Equivariant matroids}\label{sec:eq}
An {\bf equivariant matroid} $\Gamma\curvearrowright M$ consists of a finite group $\Gamma$, a matroid $M$ with ground set $\cI$,
and an action of $\Gamma$ on $\cI$ that takes flats of $M$ to flats of $M$.  In \cite{GPY}, we defined the Kazhdan-Lusztig polynomial
$P_M^\Gamma(t)$ of an equivariant matroid\footnote{In \cite{GPY}, we always denoted our group by $W$.
Here we use the letter $\Gamma$ to avoid conflict with our notation for Whitney numbers.} 
$\Gamma\curvearrowright M$.  This is a polynomial whose coefficients are virtual representations
of $\Gamma$; equivalently, it is a graded virtual representation.  If we forget the action of $\Gamma$ and take the graded dimension,
we recover the ordinary Kazhdan-Lusztig polynomial of $M$.

All of the material in Sections \ref{sec:def} and \ref{sec:KL} generalizes easily to equivariant matroids, starting with the definition of the $Z$-polynomial.
Let $L$ denote the lattice of flats of $M$.
For any flat $F\in L$, let $\Gamma_F\subset \Gamma$ denote the stabilizer of $M$.  We may then define
$$Z_M^\Gamma(t) := \bigoplus_{[F]\in L/\Gamma}t^{\rk F}\Ind_{\Gamma_F}^\Gamma P_{M^F}^{\Gamma_F}(t).$$
The generalization of Theorem \ref{closed} comes from interpreting $r$-Whitney numbers
as permutation representations.  More precisely, given an equivariant matroid $\Gamma\curvearrowright M$ and a sequence
of integers $i_r,\ldots,i_1$, let
$W_M^\Gamma(i_r,\ldots,i_1)$ be the representation of $\Gamma$ with basis $$\big\{(F_r,\ldots,F_1)\in L^r\mid F_r\leq\cdots\leq F_r\;\;\text{and $\crk F_j = i_j$ for all $j$}\big\}.$$  We omit the proof of the following result, as it does not differ significantly from the proof of Theorem \ref{closed}.

\begin{theorem}\label{closed-eq}
For all $i>0$, we have
$$c_M^\Gamma(i) \;\;\cong\;\; \sum_{r=1}^i\;\sum_{S\subset [r]} (-1)^{|S|} \sum_{\substack{a_0<a_1<\cdots < a_r<a_{r+1}\\ a_0 = 0 \\ a_r = i\\ a_{r+1} = \rk-i}}
 W_M^\Gamma\Big(a_{t_{r}(S)} + a_{r-1},\ldots,a_{t_{1}(S)} + a_{0}\Big).$$
\end{theorem}

Theorem \ref{closed-eq} takes a particularly nice form for uniform matroids.  Let $\ch_n$ be the {\bf Frobenius characteristic},
which takes representations of the symmetric group $S_n$ to symmetric functions of degree $n$ in infinitely many variables.
Let $s[n] := \ch_n \operatorname{triv}$ be the complete homogeneous symmetric function of degree $n$.

\begin{proposition}\label{uniform-whitney-eq}
We have $$\ch_{m+d} W_{U_{m,d}}^{S_{m+d}}(i_r,\ldots,i_1) = s[d-i_r]s[i_{r}-i_{r-1}]\cdots s[i_{2}-i_{1}]s[m+i_1].$$
\end{proposition}

\begin{proof}
The symmetric group $S_{m+d}$ acts transitively on the set
$$\big\{(F_r,\ldots,F_1)\in L^r\mid F_r\leq\cdots\leq F_r\;\;\text{and $\crk F_j = i_j$ for all $j$}\big\},$$
with stabilizers conjugate to the Young subgroup 
$G := S_{d-i_r}\times S_{i_{r}-i_{r-1}} \times\cdots\times S_{i_{2}-i_{1}}\times S_{m+i_1}$.
It follows that $W_{U_{m,d}}^{S_{m+d}}(i_r,\ldots,i_1)$ is isomorphic to $\Ind_G^{S_{m+d}}\triv$, and the Frobenius characteristic
of the induction of the trivial representation from a Young subgroup is equal to the product of the corresponding
complete homogeneous symmetric polynomials.
\end{proof}

\begin{corollary}\label{uni-eq}
For all $i>0$, we have
$$c_{U_{m,d}}^{S_{m+d}}(i) \cong \sum_{r=1}^i\;\sum_{S\subset [r]} (-1)^{|S|} \sum_{\substack{a_0<a_1<\cdots < a_r<a_{r+1}\\ a_0 = 0 \\ a_r = i\\ a_{r+1} = d-i}}s[m+a_{t_1(S)}]\cdot \prod_{j\in S} s[a_j - a_{j-1}] \cdot \prod_{j\in [r]\smallsetminus S} s[a_{t_{j+1}(S)} - a_{j-1}].$$
\end{corollary}

\begin{proof}
By Theorem \ref{closed-eq} and Proposition \ref{uniform-whitney-eq}, we need to show that
$$s[d-a_{t_{r}(S)} - a_{r-1}]s[a_{t_{r}(S)} + a_{r-1}-a_{t_{r-1}(S)} - a_{r-2}]\cdots s[a_{t_{2}(S)} + a_{1}-a_{t_{1}(S)} - a_{0}]s[m+a_{t_{1}(S)} + a_{0}]$$ is equal to the summand in the statement of the corollary.  First, we note that $a_0 = 0$, so the last factor is equal to $s[m+a_{t_1(S)}]$.  Next, we note that $d = a_{r+1} + a_r = a_{t_{r+1}(S)} + a_{r}$, so the first $r$ factors of the product may be written uniformly
as $$\prod_{j\in [r]}s[a_{t_{j+1}(S)} + a_{j}-a_{t_{j}(S)} - a_{j-1}].$$
For each $j\in [r]$, we have 
$$t_j(S) = \min\{\, k \mid k\geq j\;\;\text{and}\;\; k\notin S\} = \begin{cases}
j\;\;\;\;\;\;\;\;\;\;\;\;\;\text{if $j\notin S$}\\
t_{j+1}(S)\;\;\;\text{if $j\in S$},\end{cases}$$
therefore the expression
$a_{t_{j+1}(S)} + a_{j}-a_{t_{j}(S)} - a_{j-1}$ is equal to $a_{t_{j+1}(S)} - a_{j-1}$ if $j\notin S$ and $ a_{j} - a_{j-1}$ if $j\in S$.  
The result follows.
\end{proof}

\begin{remark}
A positive formula for $c_{U_{m,d}}^{S_{m+d}}(i)$ was given in \cite[Theorem 3.1]{GPY}.  It would be interesting to see if one could give
an alternative proof of that result using Corollary \ref{uni-eq}.
\end{remark}

If $V = \oplus V_i$ is a graded virtual representation of a group $\Gamma$, we say that $V$ is {\bf equivariantly log concave} if, for all $i$,
$V_i^{\otimes 2} - V_{i-1}\otimes V_{i+1}$ is isomorphic to an honest representation.  We say that $V$ is {\bf strongly equivariantly log concave}
if, for all $i\leq j\leq k\leq l$ with $i+l = j+k$, $V_j\otimes V_k - V_i\otimes V_l$ is isomorphic to an honest representation.  If $\Gamma$ is the trivial
group, then log concavity and strong log concavity are equivalent, and agree with the usual notion of log concavity for a sequence of integers.
For nontrivial $\Gamma$, however, strong equivariant log concavity is a strictly stronger condition with the desirable property of being
preserved under tensor product \cite[Remark 5.8]{GPY}.
The following conjecture is the $Z$-version of \cite[Conjecture 5.3(2)]{GPY}.  

\begin{conjecture}\label{elc}
For any equivariant matroid $\Gamma\curvearrowright M$, $Z_M^\Gamma(t)$ is strongly equivariantly log concave.
\end{conjecture}

\begin{remark}
Polynomials whose roots lie on the negative real axis are log concave in the usual sense, hence
if $\Gamma$ is the trivial group, Conjecture \ref{elc} is a weaker version
of Conjecture \ref{real}.
\end{remark}

\begin{proposition}\label{q-elc}
Fix a natural number $d$ and a prime power $q$.  Let $M$ be the matroid represented by all vectors in $\mathbb{F}_q^d$
and let $\Gamma = \operatorname{GL}_n(\mathbb{F}_q)$.
Conjecture \ref{elc} holds for $\Gamma \curvearrowright M$.
\end{proposition}

\begin{proof}
As we observed in Remark \ref{modular}, $M$ is modular, 
so the equivariant Kazhdan-Lusztig polynomial of $M$ (and of all of its contractions) is the trivial
representation in degree zero.  This means that the coefficient $z_M^\Gamma(k)$ of $t^k$ in $Z_M^\Gamma(t)$ is equal to $W^\Gamma_M(d-k)$, 
the permutation representation
on the set $G_q(d,k)$ of $k$-dimensional linear subspaces of $\mathbb{F}_q^d$.

Fix indices $i\leq j\leq k\leq l$ with $i+l = j+k$.  Let
$$S := \Big\{(V_j,V_k)\in G_q(d,j)\times G_q(d,k)\;\Big{|}\; \dim V_j\cap V_k = i\Big\}.$$
Since $S$ is a $\Gamma$-equivariant subset of $G_q(d,j)\times G_q(d,k)$, the permutation representation $\C[S]$ is naturally a direct summand of
$\C\big[G_q(d,j)\times G_q(d,k)\big].$

The map $(V_j,V_k)\mapsto (V_j\cap V_k, V_j+V_k)$ is a $\Gamma$-equivariant surjection from $S$ to $G_q(d,i)\times G_q(d,l)$.
Pulling back functions, we obtain an injection
$$z_M^\Gamma(i)\otimes z_M^\Gamma(l) \cong \C\big[G_q(d,i)\times G_q(d,l)\big] \hookrightarrow \C[S] \subset \C\big[G_q(d,j)\times G_q(d,k)\big] \cong 
z_M^\Gamma(j)\otimes z_M^\Gamma(k).$$
This completes the proof.
\end{proof}

\begin{remark}
Propositions \ref{q-real-nice} and \ref{q-elc} each strengthen in a different direction the well known fact that the polynomial
$\sum_{k=0}^d \binom{d}{k}_{\!q} t^k$ is log concave in the usual sense.
\end{remark}

\begin{remark}
The proof of Proposition \ref{q-elc} generalizes to any modular matroid.  One only has to replace $G_q(d,k)$ with the set
of flats of rank $k$, replace intersection with meet, and replace sum with join, and the proof goes through verbatim in the more
general setting.
\end{remark}

\section{Geometric interpretation}\label{sec:geom}
Let $k$ be any field, and let $V\subset \mathbb{A}_k^\cI$ be a linear subspace.  The matroid $M(V)$ on the ground set $\cI$
is characterized by the condition that $F\subset \cI$ is a flat if and only if there exists an element $v = (v_i)_{i\in \cI}\in V$ such that $F = \{i\mid v_i=0\}$.  The Kazhdan-Lusztig polynomial of $M(V)$ has a geometric interpretation \cite[Section 3]{EPW},
and a similar interpretation exists for the $Z$-polynomial, as we explain below.  This section is independent of the rest of the paper,
but Theorem \ref{z-geom} provides motivation for the definition of the $Z$-polynomial.

Let $Y(V)$ be the closure of $V$ inside of $\pon$; this variety was studied in \cite{ArBoo}\footnote{In \cite{ArBoo}, 
the authors define the matroid associated with $V$ to be the dual of the matroid that we have defined.} as well as in \cite[Section 4]{HuhWang}.
We call $Y(V)$ the {\bf Schubert variety} of $V$, in analogy with Schubert varieties in the flag variety of a semisimple algebraic group.
Let $X(V) \subset Y(V)$ be the locus where no coordinate is equal to zero.  This is called the {\bf reciprocal plane} of $V$.
The following theorem appeared in \cite[Theorem 3.10 and Proposition 3.12]{EPW}.

\begin{theorem}\label{p-geom}
If $k$ is a finite field and $\ell$ is a prime not equal to the characteristic of $k$, 
then the $\ell$-adic \'etale intersection cohomology of $X(V)$
vanishes in odd degree, and $$P_{M(V)}(t) = \sum_{i\geq 0}t^i \dim \IH^{2i}\!\big(X(V); \Ql\big).$$
If $k = \C$, the same result holds for topological intersection cohomology.
\end{theorem}

In this section we prove the analogous result for the $Z$-polynomial.

\begin{theorem}\label{z-geom}
If $k$ is a finite field and $\ell$ is a prime not equal to the characteristic of $k$, 
then the $\ell$-adic \'etale intersection cohomology of $Y(V)$
vanishes in odd degree, and $$Z_{M(V)}(t) = \sum_{i\geq 0}t^i\dim \IH^{2i}\!\big(Y(V); \Ql\big).$$
If $k = \C$, the same result holds for topological intersection cohomology.
\end{theorem}

\begin{remark}\label{pd}
In light of Theorem \ref{z-geom}, Proposition \ref{palindrome} for $M(V)$ may be interpreted as Poincar\'e duality for the intersection
cohomology of the projective variety $Y(V)$.
\end{remark}

\begin{remark}
Any matroid that can be realized over some field can be realized over a finite field, so Theorems \ref{p-geom} and \ref{z-geom}
apply to all realizable matroids.
\end{remark}

A nonempty subset $C\subset \cI$ is called a {\bf circuit} if and only if, for every flat $F$, $|C\cap F^c| \neq 1$.
Conversely, a subset $F\subset \cI$ is a flat if and only if, for every circuit $C$, $|C\cap F^c| \neq 1$.
Given a circuit $C$, there exist elements $(C_i)_{i\in C}\subset (k^\times)^C$  such that $\sum_i C_iv_i = 0$ for all $v\in V$,
and these elements are unique up to scale.
The homogeneous coordinate ring of $Y(V)\subset\pon$ has the following description \cite[Theorem 1.3(a)]{ArBoo}:
$$k[Y(V)] = k[x_i,y_i]_{i\in \cI}\Big{/}\Big{\langle}
f_C(x,y)\;\Big{|}\;\text{$C$ a circuit}
\Big{\rangle},$$
where
$$f_C(x,y) = \sum_{i\in C} C_i x_i y_{C\smallsetminus\{i\}}\and
y_S := \prod_{i\in S}y_i.$$
Given a point $p\in Y(V)$, let 
$F_p := \{i\in\cI\mid p_i\neq\infty\}.$

\begin{lemma}\label{a flat}
The set $F_p$ is a flat.
\end{lemma}

\begin{proof}
%
If $F_p$ is not a flat, then there exists a circuit $C$ and an element $i\in\cI$ such that $F_p^c\cap C = \{i\}$.
For all $j\in C\smallsetminus\{i\}$, $y_{C\smallsetminus\{j\}}$ is a multiple of $y_i$, which vanishes at $p$.  
But $x_i$ does not vanish at $p$, nor does $y_{C\smallsetminus\{i\}}$.  This contradicts the fact that $f_C$ vanishes at $p$.
\end{proof}

For any flat $F$, let $V^F\subset\mathbb{A}^{F^c}_k$ be the intersection of $V$ with $\mathbb{A}^{F^c}_k$ inside of $\mathbb{A}_k^\cI$,
and let $V_F\subset \mathbb{A}^F_k$ be the image of $V$ along the projection from $\mathbb{A}_k^\cI$.
Concretely, $V_F$ is cut out of $\mathbb{A}^F_k$ by the linear equations $f_C(x,1)$ for all circuits $C\subset F$.
Then we have $M(V^F) = M(V)^F$ and $M(V_F) = M(V)_F$.
Let $Y(V)_F := \{p\in Y(V)\mid F_p = F\}$, so that $Y(V) = \bigsqcup_{F} Y(V)_F$.

\begin{lemma}\label{vec}
For any flat $F$, there is a canonical isomorphism $Y(V)_F \cong V_F$.
\end{lemma}

\begin{proof}
The affine coordinate ring of $Y(V)_F$ is obtained from $k[Y(V)]$ by setting $x_i = 1$ and $y_i = 0$ for all $i\in F^c$ 
and $y_j=1$ for all $j\in F$.
This ring is isomorphic to $$k[x_i]_{i\in F}\Big{/}\Big{\langle}
f_C(x,1)\;\Big{|}\;\text{$C\subset F$ a circuit}
\Big{\rangle}.$$
As observed above, these are exactly the equations that define $V_F$ inside of $\mathbb{A}^F_k$.
\end{proof}

Fix a prime $\ell$ different from the characteristic of $k$.  The $\ell$-adic \'etale intersection cohomology group of $Y(V)$
is defined as $$\IH^{*}\!\big(Y(V); \Ql\big) := \H^{*-\dim Y(V)}\!\big(Y(V); \IC_{Y(V)}\big).$$
For any point $p\in Y(V)$, we define $$\IH^{*}_p\!\big(Y(V); \Ql\big) := \H^{*-\dim Y(V)}\!\big(\IC_{Y(V),p}\big)$$
to be the cohomology of the stalk of the IC sheaf at $p$.

\begin{lemma}\label{local}
For any point $p\in Y(V)$, $\IH^{*}_p\!\big(Y(V); \Ql\big)$ is isomorphic to $\IH^*\!\big(X(V^{F_p}); \Ql\big)$.
\end{lemma}

\begin{proof}
Since the IC sheaf is locally constant along strata, we may assume that $p_i\neq 0$ for all $i$, which means
that $p$ lies in the open subscheme $X(V)\subset Y(V)$.  The result then follows from the analogous statement
for $X(V)$, which is proved in \cite[Lemma 3.8]{EPW}.
\end{proof}

\begin{proofz}
We follow a slightly modified version of the argument in \cite[Section 3]{PWY}.  For any flat $F$,
let $\iota_F:Y(V)_F\to Y(V)$ be the inclusion of the stratum indexed by $F$.
There is a first quadrant cohomological spectral sequence $E$ with
$$E_1^{p,q} = \bigoplus_{\crk F = p}\mathbb{H}^{p+q}\big(\iota_F^!\operatorname{IC}_{Y(V)}\big)
\and
\bigoplus_{p+q=m} E_\infty^{p,q} = \IH^m\!\big(Y(V); \Ql\big)\;\;\text{for all $m$ \cite[\S 3.4]{BGS96}}.$$
By Lemmas \ref{vec} and \ref{local} and Poincar\'e duality,
$$\mathbb{H}^{p+q}\big(\iota_F^!\operatorname{IC}_{Y(V)}\big)
\cong 
\IH^{p+q}_c\!\big(X(V^{F}); \Ql\big) \cong \IH^{p - q}\!\big(X(V^{F}); \Ql\big).$$
We know that $\IH^{p - q}\!\big(X(V^{F}); \Ql\big)$ vanishes unless $p-q$ is even \cite[Proposition 3.9]{EPW}.
This implies that the spectral sequence degenerates at the $E_1$ page, $\IH^m\!\big(Y(V); \Ql\big) = 0$
unless $m$ is even, and
$$\IH^{2i}\!\big(Y(V); \Ql\big) \cong \bigoplus_{p+q=2i}\bigoplus_{\crk F = p}\IH^{p - q}\!\big(X(V^{F}); \Ql\big)
= \bigoplus_F\IH^{2(\crk F-i)}\!\big(X(V^{F}); \Ql\big).$$
We now apply Poincar\'e duality for $\IH^*\!\big(Y(V); \Ql\big)$ to see that we can replace $i$ with $\rk M - i$, 
which has the effect of replacing $\crk F - i$ with $i-\rk F$.
Thus $$\IH^{2i}\!\big(Y(V); \Ql\big) \cong \bigoplus_F\IH^{2(i-\rk F)}\!\big(X(V^{F}); \Ql\big).$$
By Theorem \ref{p-geom}, $\dim \IH^{2(i-\rk F)}\!\big(X(V^{F}); \Ql\big) = c_{M(V^F)}(i-\rk F) = c_{M(V)^F}(i-\rk F)$, thus
\begin{eqnarray*}\sum_{i\geq 0}t^i\dim \IH^{2i}\!\big(Y(V); \Ql\big) &=& 
\sum_{i\geq 0}t^i\sum_Fc_{M(V)^F}(i-\rk F)\\ &=&
\sum_Ft^{\rk F}\sum_{i\geq 0}c_{M(V)^F}(i-\rk F) t^{i-\rk F}\\ &=& 
\sum_Ft^{\rk F} P_{M(V)^F}(t)\\ &=&
Z_{M(V)}(t).\end{eqnarray*}
The same argument works for topological intersection cohomology
when $k=\C$.
\end{proofz}
\vspace{\baselineskip}

\begin{remark}
Theorems \ref{p-geom} and \ref{z-geom} also hold equivariantly.  That is, if 
$\Gamma$ acts on $\cI$ in such a way so that $V\subset k^\cI$ is a subrepresentation, 
then $\Gamma$ acts on $M(V)$, $X(V)$, and $Y(V)$, and we have 
$$P^\Gamma_{M(V)}(t) \cong \bigoplus_{i\geq 0} t^i \IH^{2i}\!\big(X(V))\and
Z^\Gamma_{M(V)}(t) \cong \bigoplus_{i\geq 0} t^i \IH^{2i}\!\big(Y(V))$$
as graded representations of $\Gamma$.
This holds for $\ell$-adic intersection cohomology when $k$ is a finite field as well as 
for topological intersection cohomology when $k=\C$.

The first statement for $k=\C$ appears in \cite[Corollary 2.12]{GPY}; see also \cite[Theorem 3.1]{fs-braid}.
The finite field version can be proved similarly; the only technical point is that in the $k=\C$ case we argue
that the maps in a certain spectral sequence\footnote{Here we refer to the spectral sequence $\tilde E$ that appears
in \cite[Section 3]{PWY}.} must strictly preserve weights in the mixed Hodge filtration,
and in the finite field version we instead use the fact that these maps are equivariant for the action of the Frobenius
automorphism.  

Once we know the first statement, the proof of Theorem \ref{z-geom}
extends without modification to the equivariant setting, and the second statement is proved, as well.
\end{remark}

\begin{remark}
Consider the category $\cO(V)$ of perverse sheaves on $Y(V)$ that are smooth with respect to the stratification
described in this section.  This category has some very nice properties; 
see for example \cite[3.3.1]{BGS96} when $k=\C$ and \cite[4.4.4]{BGS96} when $k$ is a finite field.
In particular, $$P_{M(V)}(t) = \sum_{i\geq 0}t^i \dim \IH^{2i}\!\big(X(V); \Ql\big) = \sum_{i\geq 0}t^i \dim \IH_\infty^{2i}\!\big(Y(V); \Ql\big),$$
which in turn is given by the (backward) graded dimension of the Ext group from the skyscraper sheaf at the point $\infty$
to the IC sheaf of $Y(V)$.  Other Ext groups from standard objects to simple objects
are measured by Kazhdan-Lusztig polynomials of localizations of contractions of $M(V)$.
\end{remark}

\bibliography{./symplectic}
\bibliographystyle{amsalpha}

\end{document}